\documentclass[11pt,leqno]{amsart}
\usepackage{amsmath,amsfonts,amssymb,amscd,amsthm,amsbsy,upref}

\renewcommand{\preceq}{\preccurlyeq}
\renewcommand{\succeq}{\succcurlyeq}
\renewcommand{\ge}{\geqslant}
\renewcommand{\geq}{\geqslant}
\renewcommand{\le}{\leqslant}
\renewcommand{\leq}{\leqslant}

\def\hangbox to #1 #2{\vskip3pt\hangindent #1\noindent \hbox to #1{#2}$\!\!$}

\newtheorem{thm}{Theorem}[section]
\newtheorem{lem}[thm]{Lemma}
\newtheorem{cor}[thm]{Corollary}
\newtheorem{prop}[thm]{Proposition}
\theoremstyle{definition}
\newtheorem{ex}[thm]{Example}

\newtheorem{defin}[thm]{Definition}

\theoremstyle{remark}
\newtheorem{rem}[thm]{Remark}
\newtheorem{rems}[thm]{Remarks}

\def\N{{\mathbb N}}
\def\R{{\mathbb R}}

\def\cT{{\mathcal T}}

\def\cB{{\mathcal B}}

\def\cF{{\mathcal F}}

\def\cN{{\mathcal N}}
\def\cU{{\mathcal U}}

\def\cV{{\mathcal V}}

\newcommand\kin{\!\in\!}

\newcommand{\fw}{\text{\fw}}

\newcommand{\ie}{{\it i.e.}}

\def\vp{\varepsilon}

\newcommand{\Mt}{\widetilde M}
\newcommand{\Nt}{\widetilde N}
\newcommand{\Pt}{\widetilde P}

\newcommand{\Frag}{\ensuremath{\mathrm{Frag}}}
\newcommand{\diam}{\ensuremath{\mathrm{diam}}}
\newcommand{\Sz}{\ensuremath{\mathrm{Sz}}}
\newcommand{\Ord}{\text{\rm Ord}}
\newcommand{\CB}{\text{\rm CB}}

\title{Generalization of a theorem of Zippin}
\begin{document}

\author{P.~Hajek}
\address{Institute of Mathematics, Academy of Science of the Czech
 Republic,  \v Zitn\'a 25 115 67 Prague~1, Czech Republic, and
 Faculty of Electrical Engineering, Czech Technical University in
 Prague,  Zikova 4, 166 27, Prague}
\email{hajek@math.cas.cz}

\author{Th.~Schlumprecht}
\address{Department of Mathematics, Texas A\&M University, College
 Station, TX 77843, USA and Faculty of Electrical Engineering, Czech
 Technical University in Prague,  Zikova 4, 166 27, Prague}
\email{schlump@math.tamu.edu}

\author{A.~Zs\'ak}
\address{Peterhouse, Cambridge, CB2 1RD, UK}
\email{a.zsak@dpmms.cam.ac.uk}
\thanks{The first named author was supported by GA\v CR 16-07378S
 and RVO: 67985840. The second named author was supported by the
 National Science Foundation under the Grant Number  DMS--1464713.}

\begin{abstract}
 We generalize  and prove a result which was first shown by
 Zippin~\cite{Zi}, and was explicitly formulated by Benyamini
 in~\cite{Be}.
\end{abstract}

\maketitle

\section{Introduction}
\label{S:1}

In 1977, Zippin~\cite{Zi} proved a result, which was reformulated by
Benyamini~\cite[page 27]{Be} as follows. For a Banach space $X$ and
$\vp>0$ we denote by $\Sz(X,\vp)$ the $\vp$-Szlenk index of $X$ whose
definition will be recalled at the end of Section~\ref{S:3}.

\begin{thm}
 \label{T:1.1}
 Let $X$ be a Banach space with separable dual, and $\vp>0$. Let $F$
 be a $w^*$-closed totally disconnected, $(1-\vp)$-norming subset of
 $B_{X^*}$, the unit ball of $X^*$.

 Then there is a countable ordinal $\alpha<\omega^{\Sz(X,\vp/8)+1}$,
 and a subspace $Y$ of $C(F)$, which is isometrically isomorphic to
 $C[0,\alpha]$, such that for every $x\in X$ there is a $y\in Y$ with
 $\|i_F(x)-y\|\le \vp(1-\vp)^{-1}\|i_F(x)\|$, where $i_F: X\to C(F)$
 denotes the embedding defined by $i_F(x)(f)=f(x)$ for $x\in X$ and
 $f\in F$.
\end{thm}

The goal of our paper is to prove a generalization of this theorem
which includes non-separable Banach spaces, and at the same time we
provide a more conceptual proof of Zippin's result (see
Theorem~\ref{T:5.1} and Corollary~\ref{C:5.2} in Section~\ref{S:5}).

The paper is organized as follows. In Section~\ref{S:1} we recall the
definition of trees, and introduce the tree topology defined on the
set $[T]$ of all branches of a tree $T$. This topology is generated by
a basis consisting of clopen sets and, if $T$ has finitely many roots,
$[T]$ with this topology is compact. Fragmentation indices of
topological spaces with respect to a pseudo-metric will be recalled in
Section~\ref{S:3}. As a particular example we define the Szlenk index
of a Banach space at the end of Section~\ref{S:3}. Section~\ref{S:4}
represents the heart of the proof of our main result (see
Theorem~\ref{T:4.2}) which might be interesting in its own right. It
is shown that if $T$ is a tree and $d$ is a pseudo-metric on $[T]$ so
that $[T]$ is fragmentable with respect to $d$, then for every $\vp>0$
there is a subset $\cN$ of the basis of the topology of $[T]$ which
forms a well-founded tree with respect to containment so that $\cN$
(in this case equal to $[\cN]$) with its tree topology is a quotient
of $[T]$, and so that  for all $M\in\cN$ the $d$-diameter of
$M\setminus\bigcup_{N\subsetneq M, N\in\cN} N$ is smaller than
$\vp$. In Section~\ref{S:5} we present and prove the generalization of
Zippin's theorem.

\section{Trees and the tree topology}
\label{S:2}

Let $T$ be a tree, which means that $T$ is a set with a reflexive
partial order denoted by $\preceq$ (where we introduce the notation
$x\prec y \iff  x\preceq y$ and $x\neq y$), with the property that for
each $t\in T$, the \emph{set of predecessors of $t$},
\[
b_t=\{ s\in T:\, s\preceq t\}
\]
is finite and linearly ordered. An \emph{initial node }of $T$ is a
minimal element of $T$, \ie, an element $t\in T$ for which
$b_t=\{t\}$. If $t\in T$ is not an initial node, then
$b_t\setminus\{t\}$ is not empty and we call the maximum of
$b_t\setminus \{t\}$  the \emph{direct predecessor of $t$.} A
\emph{successor of $u\in T$ }is an element $w\in T$ so that $u\prec w$
and it is called a \emph{direct successor of }$u$ if, moreover, there
is no $v\in T$ with $u\prec v\prec w$, in other words if $u$ is the
direct predecessor of $w$. The set of all direct successors of $u$ is
denoted by $S_u$.

A \emph{branch of $T$ }is a non-empty, linearly ordered subset of $T$
which is closed under taking direct predecessors.

\begin{rem}
 We note that for a branch $b$ it follows that if $t\in b$ and
 $s\prec t$ then also $s\in b$, and
 \begin{itemize}
 \item
   either $b$ is finite, in which case  there is a unique $t$ so that
   $b=b_t =\{ s\in T:\,s\preceq t\}$ (take $t$ to be the maximal
   element of $b$),
 \item
   or $b$ is infinite, in which case $b$ is a maximal linearly
   ordered set in $T$.
 \end{itemize}
 Indeed, if $b$ is infinite, then we can find
 $t_j\in b$, $j\in\N$, with $t_1\prec t_2\prec t_3\prec\dots$
 (choose $t_1\in b$ arbitrarily, then $t_2\in b\setminus b_{t_1}$,
 then $t_3\in b\setminus b_{t_2}$, etc.). We claim that
 $b=\bigcup_{j=1}^\infty b_{t_j}$. Indeed, since $b$ is closed under
 taking predecessors, we have
 $\bigcup_{j=1}^\infty b_{t_j}\subset b$, and if there were an element
 $t\in b\setminus\bigcup_{j=1}^\infty b_{t_j}$, then $t_j\prec t$ for
 all $j\in\N$, and thus $b_t$ would be infinite, which is a
 contradiction.  Finally we note that $\bigcup_{j=1}^\infty b_{t_j}$
 is a maximal linearly ordered set, because if for some
 $t\in T\setminus\bigcup_{j=1}^\infty b_{t_j}$ the set
 $\bigcup_{j=1}^\infty b_{t_j}\cup \{t\}$ were also linearly
 ordered, it would follow as before that $t\succ t_j$ for all
 $j\in\N$, and would contradict the assumption that $b_t$ is
 finite. Conversely, if $(t_j)$ is any strictly increasing sequence
 in $T$ then $b=\bigcup_{j=1}^\infty b_{t_j}$ is an infinite branch.
\end{rem}

We will identify a finite branch $b$ of $T$ with the element $t\in T$
such that $b=b_t$, and hence identify the tree $T$ with the set of all
its finite branches. We call the set of all infinite branches the
\emph{boundary of $T$ }and denote it by $\partial T$. If $s\in T$ and
$b\in \partial T$, we also write $s\prec b$ if $s\in b$. The set of
all branches of $T$ is denoted by $[T]$. Thus, we have
$[T]=T\cup\partial T$. The tree $T$ is called \emph{well-founded }if
$\partial T=\emptyset$. We define the \emph{ordinal  index
}$\mathrm{o}(T)$ of a well-founded tree $T$ as follows. For every
subset $S$ of $T$ we put
$S'=\{ s\in S:\,s \text{ is not maximal in }S\}$. Note that since $T$
is well-founded, if $S$ is not empty then $S'\subsetneq S$. Then we
put $T^{(0)}=T$, and by transfinite  induction for any ordinal
$\alpha$ we define
\[
T^{(\alpha)}=\big(T^{(\gamma)}\big)'\text{ if $\alpha=\gamma+1$ for
 some $\gamma$, and }T^{(\alpha)}=\bigcap_{\gamma<\alpha}
T^{(\gamma)} \text{ if $\alpha$ is a limit ordinal}.
\]
Since $T$ is assumed to be well-founded, it follows that
\[
\mathrm{o}(T)=\min\big\{\alpha\in\Ord:\,T^{(\alpha)}=\emptyset\big\}
\]
exists. Note that since the sets
$T^{(\alpha)}\setminus T^{(\alpha+1)}$, $\alpha<\mathrm{o}(T)$, are
non-empty and pairwise disjoint, it follows that if $T$ is countable,
then so is $\mathrm{o}(T)$.

Let $T$ be an arbitrary tree. We define a locally compact topology on
$[T]$, the \emph{tree topology }on $[T]$, as follows. For $t\in T$ we
put
\[
U_t=\{ b\in [T]:\,t\in b\}
\]
and let the \emph{tree topology }be generated by the set
\[
\big\{U_t, [T]\setminus U_t :\,t\in T \big\}.
\]
We call $[T]$ with its tree topology the \emph{tree space of $T$}.

Note that for $s,t\in T$, either $U_{s}\subset U_t$, or
$U_t\subset U_s$, or $U_t\cap U_s=\emptyset$. Indeed, using the
properties of trees and the definition of branches, it follows that if
$s\preceq t$, then $U_t\subset U_s$, if $t\preceq s$, then $U_s\subset
U_t$, and if $s$ and $t$ are incomparable then $U_t\cap
U_s=\emptyset$. It follows that
\[
\cB=\Big\{U_t\setminus \bigcup_{j=1}^n U_{s_j}:\,n\in \N_0,\
s_1,s_2,\ldots, s_n\in S_t\Big\}\cup \{\emptyset\}
\]
is stable under taking finite intersections, and thus is a basis of
the tree topology consisting of clopen sets. We note that for $b\in
\partial T$,
\[
\cU_b=\{ U_t:\,t\in b\}
\]
is a neighbourhood basis of $b$, and for a finite branch $b=b_t$,
\[
\cU_{b}=\Big\{ U_t\setminus \bigcup_{s\in F} U_s :\,F\subset S_t\text{
 finite}\Big\}
\]
is a neighbourhood basis of $b$. In particular, if $t$ has only
finitely many direct successors, then the singleton $\{b_t\}$ is
clopen.

\begin{prop}
 \label{P:2.2}
 $[T]$ is a locally compact Hausdorff space, and $[T]$ is compact if
 and only if $T$ has finitely many initial nodes.
\end{prop}
\begin{proof}
If $b$ and $b'$ are two different branches, and, say, without loss of
generality $t\in b\setminus b'$ then $U_t$ is a neighbourhood of $b$
and $[T]\setminus U_t$ a neighbourhood of $b'$. Thus $[T]$ is
Hausdorff.

In order to show that $[T]$ is locally compact, and also the claimed
equivalence, it is enough to show that $U_t$ is compact for each
$t\in T$. Thus, let $\cU$ be an open cover of $U_t$, and assume $\cU$
has no finite subcover of $U_t$. Then there is a $t_1\succ t$ so that
$\cU$ has no finite subcover of $U_{t_1}$. Indeed, since $t\in U_t$,
one of the elements of $\cU$ must contain a subset of the form
$U_t\setminus \bigcup_{j=1}^n U_{s_j}$ with $n\in\N$ and each
$s_j\in S_t$, which means that for one of the $s_i$ it follows that
$\cU$ has no finite subcover of $U_{s_i}$. Inductively, we can find an
increasing sequence $t\prec t_1\prec t_2\prec\ldots$ so that $\cU$ has
no finite subcover of $U_{t_i}$, $i=1,2,\ldots$. Let
$b=\bigcup_{i=1}^\infty b_{t_i}$ be the branch generated by these
$t_i$. Since $b\in U_t$, there must be a $U\in\cU$ for which $b\in U$.
This implies that there must be an $s\in b$ so that
$b\in U_s\subset U$. For large enough $n\in\N$ we have $t_n\succeq s$,
and thus $U_{t_j}\subset U$ for all $j\ge n$, which is a
contradiction.
\end{proof}
We shall call a tree $T$ \emph{compact }if the corresponding
tree space $[T]$ is compact, \ie, when $T$ has finitely many initial
nodes.

\begin{rem}
 Branches of $T$ are certain subsets of $T$. Thus we can think of the
 tree space $[T]$ as a subset of $\{0,1\}^T$. The tree topology of
 $[T]$ is simply the restriction to $[T]$ of the product topology of
 $\{0,1\}^T$.
\end{rem}

We next recall the definition of the Cantor--Bendixson index of a
compact topological space $K$. For a closed set $F\subset K$ we put
\[
d(F)=\{ \xi\in F:\,\xi\text{ is not an isolated point in $F$}\}\ .
\]
We put $d_0(K)=K$. By transfinite induction we define $d_\alpha(K)$
for ordinals $\alpha$ as follows: $d_\alpha(K)=d\big(d_\gamma(K)\big)$
if $\alpha=\gamma+1$ for some $\gamma$, and
$d_\alpha(K)=\bigcap_{\gamma<\alpha}d_\gamma(K)$ if $\alpha$ is a
limit ordinal. It follows that there must be an ordinal $\alpha_0$ for
which $d_{\alpha_0}(K)=d_{\alpha_0+1}(K)=d_{\alpha_0+2}(K)=\ldots $,
and if in that case $d_{\alpha_0}(K)=\emptyset$, we define the
\emph{Cantor--Bendixson index of $K$ }to be
\[
\CB(K)=\min\{\alpha\in \Ord:\,d_\alpha(K)=\emptyset\}\ ,
\]
otherwise we put $\CB(K)=\infty$.

Now we assume that $T$ is a well-founded tree with finitely many
initial nodes, and we want to compare $\mathrm{o}(T)$ with $\CB(T)$.
Since every maximal element in any subset $S$ of $T$ is isolated in
$S$, it follows that
\begin{equation}
 \label{E:2.1}
 \CB(T)\le \mathrm{o}(T).
\end{equation}

It follows from results in general topology that if $K$ is a
non-empty, countable, compact space, then $\CB(K)=\beta+1$ for a
countable ordinal $\beta$, and $|d_\beta(K)|=n$ for some $0<n<\omega$,
and moreover $K$ is homeomorphic to the ordinal interval
$[0,\omega^\beta\cdot n]$. We next give a direct proof of this fact in
the special case when $K$ is a well-founded, countable, compact tree
with its tree topology. We shall also prove the converse that every
countable successor ordinal (and thus every countable, compact
topological space) is homeomorphic to a well-founded, countable,
compact tree.

\begin{thm}
 \label{thm:trees-vs-ordinals}
 Let $(T,\preceq)$ be a countable, well-founded tree. Then there is
 an ordinal $\beta\leq \omega^{\mathrm{o}(T)}$ such that $T$ with its
 tree topology is homeomorphic to the ordinal interval $[0,\beta)$.

 Conversely, given a countable ordinal $\beta$, the interval
 $[0,\beta)$ is homeomorphic to a countable, well-founded tree.
\end{thm}
\begin{proof}
 For $t\in T$ let $\alpha(t)$ be the ordinal $\alpha<\mathrm{o}(T)$
 such that $t\in [T]^{(\alpha)}\setminus [T]^{(\alpha+1)}$. We first
 show that for each $t\in T$ there is an ordinal $\beta\leq
 \omega^{\alpha(t)}$ and a homeomorphism $\varphi\colon
 U_t\to[0,\beta]$ with $\varphi(t)=\beta$. We shall proceed by
 induction on $\alpha(t)$.

 If $\alpha(t)=0$, then $U_t=\{t\}$, so the claim follows. Now assume
 that $\alpha=\alpha(t)>0$. Then $S_t\neq\emptyset$ and by definition
 of $\alpha(t)$, if $s\in S_t$, then $s\notin [T]^{(\alpha)}$, and
 thus $\alpha(s)<\alpha$. Enumerate $S_t$ as a (finite or infinite)
 sequence $s_1, s_2, \dots$. By induction hypothesis, each clopen set
 $U_{s_n}$ is homeomorphic to $[0,\beta_n]$ for some ordinal
 $\beta_n\leq \omega^{\alpha(s_n)}$. Set $\gamma_0=0$ and
 $\gamma_n=\gamma_{n-1}+\beta_n+1$ for $n\geq 1$. Let
 $\beta=\sup_{n\geq 1} (\gamma_n+1)$. Since $\beta_n<\omega^\alpha$ for
 all $n$, it follows that $\beta\leq \omega^\alpha$. For each
 $n\geq 1$, the interval $[\gamma_{n-1},\gamma_n)$ is
 order-isomorphic to $[0,\beta_n]$. Since $[0,\beta)$ is the disjoint
 union of the clopen intervals $[\gamma_{n-1},\gamma_n)$, it follows
 that $\bigcup_n U_{s_n}$ is homeomorphic to $[0,\beta)$. Let
 $\varphi$ be such a homeomorphism. We extend $\varphi$ to $U_t$ by
 setting $\varphi(t)=\beta$. Under $\varphi$ the basic neighbourhood
 $U_t\setminus \bigcup_{1\leq k\leq n} U_{s_k}$ of $t$ corresponds to
 the basic neighbourhood $[\gamma_n,\beta]$ of $\beta$. Hence
 $\varphi\colon U_t\to [0,\beta]$ is a homeomorphism with
 $\varphi(t)=\beta$. This completes the proof of the induction step.

 We now finish the proof of the first half of the theorem as
 follows. We join a root to $T$ by adding a new element $r$ to $T$
 and by declaring $r\prec t$ for all $t\in T$. Put
 $T^*=T\cup\{r\}$. An easy induction shows that
 $\big( T^{(\alpha)} \big)^*= \big( T^* \big)^{(\alpha)}$ for all
 $\alpha\leq \mathrm{o}(T)$, and hence in $T^*$ we have
 $\alpha(r)=\mathrm{o}(T)$. By our initial claim, there is an ordinal
 $\beta\leq\omega^{\mathrm{o}(T)}$ and a homeomorphism $\varphi\colon
 T^*\to [0,\beta]$ with $\varphi(r)=\beta$. Thus, $T$ is homeomorphic
 to the interval $[0,\beta)$, as required.

 For the converse statement, it is enough to prove that for all
 $\beta<\omega_1$ there is a countable, well-founded tree
 $(T,\preceq)$ with one initial node $r$ and a homeomorphism
 $\varphi\colon T\to [0,\beta]$ with $\varphi(r)=\beta$. Indeed, it
 then follows that the interval $[0,\beta)$ is homeomorphic to
 $T\setminus \{r\}$ with the subspace topology which is easily seen
 to be the same as its tree topology. We proceed by induction on
 $\beta$.

 For $\beta=0$ we take $T$ to be the tree with one element. If $T$ is
 a suitable tree for some $\beta$, then $T^*$, \ie, $T$ with a new
 root adjoined, works for $\beta+1$. Let us now assume that $\beta$
 is a countable limit ordinal. Then there is an increasing sequence
 $(\gamma_n)_{1\leq n<\omega}$ of successor ordinals with
 $\beta=\sup\gamma_n$. Set $\gamma_0=0$ and for each $1\leq n<\omega$
 choose $\beta_n$ with $\gamma_n=\gamma_{n-1}+\beta_n+1$. Since
 $\beta_n<\beta$, by induction hypothesis, there is a well-founded,
 countable tree $(T_n,\preceq_n)$ with one initial node $s_n$
 homeomorphic to the interval $[\gamma_{n-1},\gamma_n)$. Let
 $(T,\preceq)$ be the disjoint union of the $T_n$ together with a new
 element $r$ such that for all $s,t\in T$ we have $s\preceq t$ if and
 only if either $s,t\in T_n$ and $s\preceq_n t$ for some $1\leq
 n<\omega$, or $s=r$. Then $T$ is a well-founded, countable tree with
 one initial node $r$, and moreover $S_r=\{s_n:\,1\leq n<\omega\}$
 and $U_{s_n}=T_n$ for each $n$. As we have seen in the proof of the
 first half of the theorem, in this situation there is a
 homeomorphism $\varphi\colon T\to [0,\beta]$ with
 $\varphi(r)=\beta$. 
\end{proof}

\begin{rems}
 Since for $\beta>0$ the interval $[0,\beta)$ is compact if and only
 if $\beta$ is a successor ordinal, it follows from
 Theorem~\ref{thm:trees-vs-ordinals} that a non-empty, countable,
 well-founded, compact tree $T$ is homeomorphic to $[0,\beta]$ for
 some ordinal $\beta<\omega^{\mathrm{o}(T)}$.

 Another consequence of the above theorem is that a countable,
 well-founded tree $(T,\preceq)$ has a well-ordering, and the
 corresponding order topology is the same as the tree topology. This
 ordering can be described explicitly as follows. Adjoin a root $r$
 to $T$, and set $T^*=T\cup\{r\}$. This allows us to refer to the set
 of $\preceq$-minimal elements of $T$ as $S_r$. For each $t\in T^*$
 fix a well-ordering $<$ of $S_t$ with order type at most
 $\omega$. We then extend $<$ to a linear ordering of $T$ as
 follows. For $s,t\in T$, we let $s<t$ if and only if either
 $t\prec s$, or there exist $u\in T^*$ and $v,w\in S_u$ such that
 $v\preceq s$, $w\preceq t$ and $v<w$.
\end{rems}

We next give an example of an uncountable compact space that can also
be realized as a tree space. This example will be important later.

\begin{ex}
 \label{Ex:2.6} Let $D$ be the {\em Cantor set}, \ie, the set
 $D=\{0,1\}^\N$ endowed with the product topology of the discrete
 topology on $\{0,1\}$. Denote by $[\N]$, $[\N]^{<\omega}$ and
 $[\N]^{\omega}$ the subsets of $\N$, the finite subsets of $\N$, and
 the infinite subsets of $\N$, respectively. Each element of $D$ can
 be identified in a canonical way with an element of $[\N]$, and vice
 versa, by letting $A_\sigma=\{n: \sigma_n=1\}$ for
 $\sigma=(\sigma_n)\in D$. Via this identification $[\N]$ becomes a
 compact, metrizable space.

 Now we will show that this topology on $[\N]$ is identical with the
 tree topology on the branches of a tree $T$. Indeed, let
 $T=[\N]^{<\omega}$, on which we consider the tree structure given by
 extension.  For $A=\{a_1,a_2,\dots , a_m\}$ and
 $B=\{b_1, b_2, \dots, b_l\}$, both sets written in increasing order,
 we say that \emph{$B$ is an extension of $A$}, or that \emph{$A$ is
   an initial segment of $B$}, and write $A\prec B$, if $l>m$ and
 $a_i=b_i$ for $i=1,2,\dots ,m$.  This turns $T$ into a tree whose
 only initial node is $\emptyset$ and it is easy to see that
 $[T]=[\N]$ and $\partial T=[\N]^{\omega}$.  Here we identify any
 $A=\{a_1,a_2,\dots \}\in [\N]^{\omega}$ (written in increasing
 order) with the branch
 $b=\big\{ \{a_1,a_2,\ldots ,a_n\}:\,n=0,1,2,\dots \big\}$.

 If $A\in[\N]^{<\omega}$, then  $S_A=\{ A\cup\{n\}:\,n>\max(A)\}$,
 and $U_A=\{ B\in[\N] \ :\,B\succeq A\}$. For $n\ge \max(A)$ we put
 \[
 U_{A,n}=\{ C\succeq A:\,C\setminus A\subset [n+1,\infty)\}
 \]
 (thus $U_{A,\max(A)}=U_A$). Then $\cB=\{
 U_{A,n}:\,A\in[\N]^{<\omega}, n\ge \max(A)\}$ is a basis of the
 product topology consisting of clopen sets.

 By definition, $U_A$ is also clopen in the tree topology. If
 $\cF\subset S_A$ is finite, then
 \begin{align*}
   U_A\setminus\bigcup_{B\in\cF} U_B&=\{ D\succeq A:\,\forall
   B\in \cF\quad \max(B)\neq \min (D\setminus A)\}\ .
 \end{align*}
 Thus, if we put $N=\max\big\{\max(B) :\,B\in \cF\big\}$,
 it follows that
 \[
 U_{A,N}=U_A\setminus \bigcup_{n=\max(A)+1}^N U_{A\cup\{n\}}\subset
 U_A\setminus\bigcup_{B\in\cF} U_B\subset U_A\ .
 \]
 This implies that the product topology on $[\N]$ and the tree
 topology coincide.
\end{ex}

We conclude this section with a well-known folklore result in
topology. For the convenience of the reader we include the proof.

\begin{lem}
 \label{lem:general-topology}
 Let $K$ be a compact Hausdorff space, $\vp>0$, and
 $f_1\colon K\to \R$ a function such that every point of $K$ has a
 neighbourhood on which the oscillation of $f_1$ is at most
 $\vp$. Then there is a continuous function $f\colon K\to\R$ such
 that $|f(x)-f_1(x)|\leq\vp$ for all $x\in K$.
\end{lem}
\begin{proof}
 By assumption, the family of open subsets of $K$ on which the
 oscillation of $f_1$ is at most $\vp$ is an open cover for $K$, and
 hence contains a finite subcover $U_1,\dots,U_n$. Let
 $\varphi_1,\dots,\varphi_n$ be a partition of unity subordinate to
 the cover $U_1,\dots,U_n$. Thus, $\varphi_j\colon K\to [0,1]$ is a
 continuous function whose support is contained in $U_j$ for each
 $j=1,\dots,n$ such that
 $\sum_{j=1}^n\varphi_j(x)=1$ for all $x\in K$. For each
 $j=1,\dots,n$ fix $x_j\in U_j$, and define $f\colon K\to\R$ by
 setting $f(x)=\sum_{j=1}^n f_1(x_j)\varphi_j(x)$, $x\in K$. Then $f$
 is continuous and
 \[
 |f(x)-f_1(x)| \leq \sum_{j=1}^n \varphi_j(x) | f_1(x_j)-f_1(x) | \leq
 \vp\qquad \text{for all }x\in K\ ,
 \]
 as required.
\end{proof}

\section{Fragmentation indices}
\label{S:3}

In this section we recall some well known notation and results on
fragmentation indices.  All of the results below, and much more, may
be found in books on topology and descriptive set theory (for
example~\cite{Do}). Nevertheless, for better reading, we would like to
recall the results we will need here. We also do this because we
consider fragmentations of topological spaces with respect to
pseudo-metrics, and not only metrics.

\begin{defin}
 \label{D:3.1}
 Let $(X,\cT)$ be a topological space and $d(\cdot,\cdot)$ a
 pseudo-metric on $X$. We say that $(X,\cT)$ is
 \emph{$d$-fragmentable }if for all non-empty closed subsets $F$ of
 $X$ and all $\vp>0$ there is an open set $U\subset X$ so that
 $U\cap F\neq\emptyset$ and $d$-$\diam(U\cap F)<\vp$.
\end{defin}

The following statement is a well known corollary of the Baire
Category Theorem.
\begin{thm}
 \label{T:3.2}
 Let $(X,\cT)$ be a Polish space (\ie, separable and completely
 metrizable), and let $d$ be a pseudo-metric on $T$ so that all
 closed $d$-balls, $B_r(x)= \{y\in X: d(x,y)\le r\}$ with $r>0$ and
 $x\in X$, are closed in $X$ with respect to $\cT\!\!$.

 Then $(X,\cT)$ is $d$-fragmentable if and only if $(X,d)$ is
 separable.
\end{thm}
\begin{proof}
 ``${\Leftarrow}$" Let $F\subset X$ be $\cT$-closed and
 $\vp>0$. Choose $D\subset F$ dense in $(F,d)$ and countable, and
 then note that $F$ can be written as the countable union of
 $\cT$-closed sets in the following way:
 \[
 F=\bigcup_{a\in D} F\cap B_\vp(a)\ .
 \]
 By the Baire Category Theorem there must therefore be an $a\in D$ so
 that $F\cap B_\vp(a)$ has a non-empty interior with respect to the
 subspace topology defined by $\cT$ on $F$.

 \noindent
 ``${\Rightarrow}$" Assume that $(X,d)$ is not separable.  We need to
 find $\vp>0$ and a $\cT$-closed set $F$ in $X$ which has the
 property that $d$-$\diam(U\cap F)\ge \vp$ for all $\cT$-open
 $U\subset X$ with $U\cap F\not=\emptyset$.

 Since $(X,d)$ is not separable, we find an uncountable $A\subset X$
 and an $\vp>0$ so that $d(x,z)>\vp$ for all $x\not= z$ in $A$. Set
 \[
 B=\{ x\in A :\,\exists\ U\in\cT \text{ such that } x\in U \text{ and
 } U\cap A \text{ is countable}\}\ .
 \]
 Then $\cU=\{U\in\cT :\,U\cap A \text{ countable}\}$ is an open cover
 for $B$, and hence it has a countable subcover $\cV$. (Here we are
 using the fact that a Polish space is second countable, and hence so
 are all its subsets.) It follows that
 $B=\bigcup_{V\in\cV} (B\cap V)$ is countable. Consider the
 $\cT$-closed subset $F=\overline{A\setminus B}^{\cT}$ of $X$. Let
 $U$ be a $\cT$-open subset of $X$ with $U\cap F\neq\emptyset$. Then
 $U\cap (A\setminus B)\neq\emptyset$, and hence $U\cap (A\setminus
 B)$ is uncountable by the definition and countability of $B$. It
 follows that $U\cap (A\setminus B)$ has at least two elements which,
 by definition of $A$, implies that $d$-$\diam (U\cap F)\geq\vp$.
\end{proof}

\begin{defin}
 \label{D:3.3}
 Let $(X,\cT)$ be a topological space and $d(\cdot,\cdot)$ a
 pseudo-metric on $X$. For a closed set $F\subset X$ and $\vp>0$ we
 define the \emph{$\vp$-derivative of $F$ }by
 \begin{align*}
   F'_{\vp}&=F\setminus\bigcup\{ U:\,U\subset X\text{ is $\cT$-open
     and } d\text{-}\diam(U\cap F)<\vp\}\\ &=\big\{ \xi\in
   F:\,\forall\,U\kin \cU_\xi\ d\text{-}\diam(U\cap F)\geq
   \vp\big\}\ ,
 \end{align*}
 where $\cU_\xi$ is the set of all open neighbourhoods of $\xi$ (with
 respect to $\cT$). For every ordinal $\alpha$ we define the
 \emph{$\vp$-derivative of $F$ of order $\alpha$}, denoted
 $F^{(\alpha)}_\vp$, by transfinite induction:
 \begin{align*}
   &F^{(0)}_\vp=F, \quad
   F^{(\alpha)}_\vp=\big(F^{(\gamma)}_\vp\big)'_\vp\text{ if
     $\alpha=\gamma+1$, and } F^{(\alpha)}_\vp =\bigcap_{\gamma<\alpha}
   F^{(\gamma)}_\vp \text{ if $\alpha$ is a limit ordinal.}
 \end{align*}
\end{defin}

Let $(X,\cT)$ be a topological space, $d(\cdot,\cdot)$ a pseudo-metric
on $X$, $F$ be a $\cT$-closed subset of $X$, and $\vp>0$. First we
note that if $F^{(\alpha)}_\vp=F^{(\alpha+1)}_\vp$ then
$F^{(\alpha)}_\vp=F^{(\beta)}_\vp$ for all $\beta>\alpha$. It follows
that if $F^{(\alpha)}_\vp\neq F^{(\alpha+1)}_\vp$, then
$\beta\mapsto F^{(\beta)}_\vp\setminus F^{(\beta+1)}_\vp$ defines an
injection on $\alpha$ into the power set of $F$, which is not possible
for $\alpha$ sufficiently large. Therefore there must be a minimal
ordinal $\alpha_0$ for which
\[
F^{(\alpha_0)}_\vp=F^{(\alpha_0+1)}_\vp=F^{(\alpha_0+2)}_\vp=\dots\ .
\]
We put $F^{(\infty)}_\vp=F^{(\alpha_0)}_\vp$. If $(X,\cT)$ is
$d$-fragmentable, then $F^{(\infty)}_\vp=\emptyset$.

We define the {\em $\vp$-fragmentation index of $F$ with respect to
 $d$ }by
\[
\Frag(d,F,\vp)=\Frag(F,\vp)=%
\begin{cases}
 \min\{ \beta\in \Ord:\,F^{(\beta)}_\vp=\emptyset\} & \text{if
   $F^{(\infty)}_\vp=\emptyset$,}\\
 \infty &\text{if not.}
\end{cases}
\]
Here we consider ``$\infty$'' to be outside of the class of
ordinals. Secondly, we define the \emph{fragmentation index of $F$
 with respect to $d$ }by
\[
\Frag(d,F)=\Frag(F)=\sup_{\vp>0} \Frag(F,\vp)
\]
with $\Frag(F)=\infty$ if for some $\vp>0$ we have
$\Frag(F,\vp)=\infty$. 
\begin{rem}
Assume that $(X,\cT)$ is $d$-fragmentable and that $F\subset X$ is
compact. Let $\vp>0$. Then it follows for a limit ordinal $\alpha$
for which $F^{(\gamma)}_{\vp}\not=\emptyset$ whenever $\gamma<\alpha$
that also $F^{(\alpha)}_{\vp}\not=\emptyset$. Therefore
$\Frag(d,F,\vp)$ will always be a successor ordinal.

If $(X,\cT)$ is a Polish space (and thus second countable) that is
$d$-fragmentable, then for a closed $F\subset X$ and $\vp>0$ it
follows from the fact that
$F^{(\beta)}_\vp\subsetneq F^{(\alpha)}_\vp$ for 
$\alpha<\beta\le \Frag(F,\vp)$ and from~\cite[Theorem 6.9]{Ke} that
$\Frag(F,\vp)<\omega_1$, where $\omega_1$ denotes the first
uncountable ordinal, and thus also that $\Frag(F)<\omega_1$.
\end{rem}
As an important example we consider the Szlenk index of a Banach
space, which we will introduce for general Bananch spaces, not only
for separable ones.  We call a Banach space $X$ an \emph{Asplund space
}if every separable subspace of $X$ has separable dual. This is not
the original definition of Asplund~\cite{As}, but proven to be
equivalent to it. The following equivalence is stated in~\cite{FHHMZ}
and gathers the results from~\cite{As,NP,Phe}.

\begin{thm}\cite[Theorem~11.8, p.~486]{FHHMZ}
 \label{T:3.6}
 Let $(X,\|\cdot\|)$ be a Banach space.  Then the following
 assertions are equivalent:

 {\em (i)}  $X$ is an Asplund space.

 {\em (ii)} $B_{X^*}$ with the $w^*$-topology is
 $\|\cdot\|$-fragmentable. Here $\|\cdot\|=\|\cdot\|_{X^*}$ denotes
 the dual norm on $X^*$, and $\|\cdot\|$-fragmentable means
 $d$-fragmentable, where $d$ is the induced metric defined by
 $d(x^*,y^*)=\|x^*-y^*\|$ for $x^*,y^*\in X^*$.
\end{thm}

Assume that $X$ is an arbitrary Banach space. For a $w^*$-closed
subset $F$ of $B_{X^*}$ and $\vp>0$ we denote the $\vp$-fragmentation
index of $F$ with respect to $\|\cdot\|_{X^*}$ by $\Sz(F,\vp)$ and
call it the $\vp$-\emph{Szlenk index of $F$}. The \emph{Szlenk index
 of $F$ } is then $\Sz(F)=\sup_{\vp>0} \Sz(F,\vp)$.  The
\emph{$\vp$-Szlenk index of }$X$ is then defined to be
$\Sz(B_{X^*},\vp)$ and denoted by $\Sz(X,\vp)$, and the \emph{Szlenk
 index of $X$ }is
$\Sz(X)=\sup_{\vp>0} \Sz(X,\vp)=\Sz(B_{X^*})$. Note that by
Theorem~\ref{T:3.6} above, $X$ is an Asplund space if and only if all
these indices are ordinal numbers.

\begin{rem}
 \label{rem:szlenk-of-norming-set}
 Let $K$ be a compact topological space. By identifying the elements
 of $K$ with their Dirac measure, we can think of $K$ as a compact
 subset of $B_{C(K)^*}$ which 1-norms the elements of $C(K)$.  It is
 then easy to see that $\CB(K)=\Sz(K,\vp)=\Sz(K)$ for all $0<\vp<2$.
 It follows therefore that $\CB(K)=\Sz(K)\le \Sz(C(K))$. In general
 it is not true that $\Sz(K)=\Sz(C(K))$. Nevertheless,
 in~\cite[Theorem C]{Sch} for the case of separable dual, and
 in~\cite[Theorem 1.1]{Ca} for the general case, it was shown that if
 $X$ is a Banach space and $B\subset B_{X^*}$ is compact and
 $1$-norming for $X$, then
 \begin{equation}
   \label{E:3.1}
   \Sz(X)=\min\big\{\omega^\alpha:\,\omega^\alpha\ge \Sz(B)\big\}
 \end{equation}
 if $X$ is an Asplund space, and $\Sz(X)=\Sz (B)=\infty$ otherwise.
\end{rem}

\section{Fragmentation of $[T]$}
\label{S:4}

Throughout this section we fix a tree $T$ and a pseudo-metric
$d(\cdot,\cdot)$ on its tree space $[T]$ which, we recall, is the set
of all branches of $T$ equipped with the tree topology. We assume that
$T$ is compact, \ie, that it has finitely many initial nodes or,
equivalently, that its tree space $[T]$ is compact. We also assume
that $[T]$ is $d$-fragmentable. This situation arises in the following
important example which we will later use.
\begin{ex}
 \label{Ex:4.1}
 Consider the space $C([T])$ of continuous functions on $[T]$ for our
 compact tree $T$. Let $X$ be a closed subspace of $C([T])$ and
 assume that $X$ is an Asplund space. For $b_1,b_2\in[T]$ set
 \[
 d(b_1,b_2)= \sup_{x\in B_X} | x(b_1) - x(b_2) |\ .
 \]
 Then $d(\cdot,\cdot)$ is a pseudo-metric on $[T]$ and the map
 sending $b\in[T]$ to the Dirac measure at $b$ restricted to $X$ is
 an isometry of $([T],d)$ into $(B_{X^*},\|\cdot\|)$. It follows from
 Theorem~\ref{T:3.6} above that $[T]$ is $d$-fragmentable.
\end{ex}

We now fix an $\vp>0$, and let $\eta$ be the ordinal so that
$\Frag(d,[T],\vp)=\eta+1$. We abbreviate
$[T]^{(\alpha)}=[T]^{(\alpha)}_\vp$ for $\alpha\in \Ord$. Let $\cB$ be
the family of basic open subsets of $[T]$, \ie, the sets of the form
$N=U_t\setminus \bigcup_{s\in F} U_s$, where $t\in T$ and $F$ is a
finite (possibly empty) subset of $S_t$. Note that $t$ and $F$ are
uniquely determined by $N$. Indeed, $t$ is the least element of $N$,
and then $F$ is the complement in $S_t$ of the set of minimal elements
of $N\setminus\{t\}$. We say $N$ is of \emph{type~I }if
$F=\emptyset$, otherwise we say $N$ is of \emph{type~II}. Note that
$\cB$ is partially ordered by containment: $M\preccurlyeq N$ if and
only if $M\supseteq N$. However, in general, $\cB$ is not a tree. The
following theorem is the main result of this section.
\begin{thm}
 \label{T:4.2}
 Let $T,d,\vp,\eta$ and $\cB$ be as above. Then there exists a subset
 $\cN$ of $\cB$ which is a well-founded tree under containment such
 that
 \[
 \text{$d$-$\diam$} \Big( M\setminus \bigcup_{N\in S_M} N \Big) <\vp
 \]
 for each $M\in\cN$ (where, as before, $S_M$ denotes the set of
 direct successors of $M$ in the tree $(\cN, \supset)$), and the
 ordinal index $\mathrm{o}(\cN)$ of $\cN$ satisfies 
 $\mathrm{o}(\cN)\leq\lambda +2n+2$, where $\eta=\lambda +n$,
 $\lambda$ is a limit ordinal and $n<\omega$. Moreover,
 $\bigcup \cN=[T]$, and $\cN$ has finitely many initial nodes.
\end{thm}

\begin{proof}
 For $b\in[T]$ let $\alpha(b)$ be the ordinal $\alpha\leq\eta$ such
 that $b\in [T]^{(\alpha)}\setminus [T]^{(\alpha+1)}$. Then $b$ has a
 neighbourhood whose intersection with $[T]^{(\alpha)}$ has
 $d$-diameter less than $\vp$. If there exists a type~I neighbourhood
 of $b$ with that property (which is the case if $b\in\partial T$),
 then there exists a least $s\in b$ such that
 $d$-$\diam\big(U_s\cap [T]^{(\alpha)}\big)<\vp$, and in this case we
 set $N_b=U_s$. Otherwise $b$ is necessarily a finite branch $b_t$
 for some $t\in T$ and
 $d$-$\diam\big(U_t\cap [T]^{(\alpha)}\big)\geq\vp$. In this case
 there is a minimal (with respect to inclusion), finite, non-empty
 subset $F$ of $S_t$ such that the
 $d$-$\diam\Big(\big(U_t\setminus \bigcup_{s\in F}U_s\big)\cap
 [T]^{(\alpha)}\Big)<\vp$. Note that $F$ is not necessarily unique:
 we simply choose one such minimal $F$ and set
 $N_b=U_t\setminus \bigcup_{s\in F}U_s$. We do this for every
 $b\in[T]$ and set $\cN =\{ N_b:\, b\in[T]\}$. For $N\in\cN$ we let
 $\alpha(N)=\alpha(b)$ where $b\in[T]$ is such that $N=N_b$. Note
 that this definition does not depend on the choice of $b$. Indeed,
 we have
 \begin{equation}
   \label{E:4.1}
   \alpha(N)=\max \big\{ \beta\leq \eta:\, N\cap
         [T]^{(\beta)}\neq\emptyset \big\}=\min \Big\{ \beta:\,
         d\text{\,-}\diam\big(N\cap [T]^{(\beta)}\big)<\vp
         \Big\}\ .
 \end{equation}
 We now prove two simple facts. Recall that we identify $t\in T$ with
 the finite branch $b_t$. So we will sometimes write $N_t$ instead of
 $N_{b_t}$.
 \begin{lem}
   \label{L:4.3}
   Let $M_1,M_2\in\cN$. Then either $M_1\subset M_2$ or
   $M_1\supset M_2$ or $M_1\cap M_2=\emptyset$.
 \end{lem}
 \begin{proof}
   Let us first note that if $N\in \cN$ is of type~II, and thus of
   the form $N = U_t\setminus \bigcup_{s\in F} U_s$ for a unique
   $t\in T$ and finite, non-empty $F\subset S_t$, then for
   $b\in [T]$ we have $N=N_b$ if and only if $b=b_t$. It follows
   that if $N_1$ and $N_2$ in $\cN$ are both of type~II and of the
   form $N_1 = U_t\setminus \bigcup_{s\in F_1} U_s$ and
   $N_2 = U_t\setminus\bigcup_{s\in F_2}U_s$, then
   $N_1=N_{b_t}=N_2$.

   For each $i=1,2$, choose $t_i\in T$ and finite
   $F_i\subset S_{t_i}$, such that
   $M_i=U_{t_i}\setminus \bigcup _{s\in F_i}U_s$ (where the $F_i$
   could be empty, and thus $M_i$ be of type~I).  If $t_1$ and $t_2$
   are incomparable, then
   $M_1\cap M_2\subset U_{t_1}\cap U_{t_2}=\emptyset$. If $t_1=t_2$
   and one of $F_1$ and $F_2$ is empty, then $M_1\subset M_2$ or
   $M_1\supset M_2$. If $t_1=t_2$ and both $F_1$ and $F_2$ are
   non-empty, then $M_1$ and $M_2$ are type~II neighbourhoods, and
   hence, by the remark at the beginning of the proof,
   $b=b_{t_1}=b_{t_2}$ is the unique branch such that $M_1=M_2=N_b$.

  Finally, assume that $t_1$ and $t_2$ are comparable and
  distinct. We may without loss of generality assume that $t_1\prec
  t_2$. Let $s$ be the unique direct successor of $t_1$ with
  $s\preccurlyeq t_2$. Then either $s\in F_1$, and thus $M_1\cap
  M_2=\emptyset$, or $s\notin F_1$, and then $M_1\supset M_2$.
 \end{proof}
 Before the next lemma, we observe the following consequence
 of~\eqref{E:4.1}. If $M, N\in\cN$ and $M\supset N$, then
 $\alpha(M)\geq \alpha(N)$.
 \begin{lem}
   \label{L:4.4}
   Let $M,N\in\cN$. Assume that $M\supsetneq N$ and
   $\alpha(M)=\alpha(N)$. Then $M$ is of type~II and $N$ is of
   type~I.
\end{lem}
 \begin{proof}
   Set $\alpha=\alpha(M)=\alpha(N)$, and choose branches $b,c\in[T]$
   such that $M=N_b$ and $N=N_c$. Assume for a contradiction that $M$
   is of type~I. Then $M=U_t$ for some $t\in b$, and so the
   $d$-diameter of $U_t\cap[T]^{(\alpha)}$ is less than
   $\vp$. Since $M\supset N$, we have $t\in c$, and thus by the
   definition of $N_c$, we have $N_c= U_s$ with
   $s=\min\big\{ r\in c:
   d\text{-}\diam(U_r\cap[T]^{\alpha})<\vp\big\}$. But this implies
   that $s=t$ and we must have $M=U_t=N$, which is a
   contradiction. Thus $b$ is a finite branch $b_t$, say, and
   $M=U_t\setminus \bigcup_{s\in F}U_s$ for some non-empty, finite
   set $F\subset S_t$. Since $M\supsetneq N$, there must be an
   $s\in S_t\setminus F$ such that $c\in U_s$. Since $U_s\subset M$,
   it follows that
   $d$-$\diam \big( U_s\cap [T]^{(\alpha)}\big)<\vp$. Hence $N=U_s$,
   and so $N$ is of type~I.
 \end{proof}

 We shall make use the following immediate consequence of
 Lemma~\ref{L:4.4}. Given $M,N,P\in\cN$, if
 $M\supsetneq N\supsetneq P$, then $\alpha(M)>\alpha(P)$.

 \noindent{\it Continuation of the proof of Theorem~\ref{T:4.2}.}
 It follows from Lemma~\ref{L:4.3} that for $M\in\cN$ the set
 $b_M=\{N\in\cN:\,N\supset M\}$ is linearly ordered.  Write $M$ as
 $M=U_t\setminus\bigcup_{s\in F} U_s$ with $t\in T$ and
 $F\subset S_t$ finite.  To see that $b_M$ is finite, observe that if
 $N_1\supsetneq N_2$ in $\cN$ with $N_i=U_{t_i}\setminus
 \bigcup_{s\in F_i}U_s$ for some $t_i\in T$ and finite
 $F_i\subset S_{t_i}$ ($i=1,2$), then either $t_1\prec t_2$, or
 $t_1=t_2$ and $F_1=\emptyset\neq F_2$. This shows that the
 cardinality of $b_M$ is at most twice the cardinality of the set of
 predecessors of $t$, and thus $\cN$ is a tree. We next verify that
 $\cN$ is well-founded. Assume that there is an infinite sequence
 $N_1\supsetneq N_2\supsetneq N_3\supsetneq \dots$ in
 $\cN$. By~\eqref{E:4.1}, we have
 $\alpha(N_1)\geq\alpha(N_2)\geq\dots$, and hence this sequence of
 ordinals is eventually constant. Lemma~\ref{L:4.4} shows that this
 is not possible.

 We will now prove the stated upper bound on o$(\cN)$. Fix a limit
 ordinal $\alpha$ and assume that
 \begin{equation}
   \label{E:4.2}
   \cN^{(\alpha)}\subset \{ N\in\cN:\,\alpha(N)\geq \alpha\}\ .
 \end{equation}
 We show by induction that $\cN^{(\alpha+2m)}\subset \{ N\in\cN:\,
 \alpha(N)\geq \alpha+m\}$ for all $m<\omega$. The case $m=0$ is our
 base assumption. Now let $M\in\cN^{(\alpha+2m+2)}$. Then
 $M\supsetneq N\supsetneq P$ for some $N\in\cN^{(\alpha+2m+1)}$ and
 $P\in\cN^{(\alpha+2m)}$. By induction hypothesis we have
 $\alpha(P)\geq\alpha+m$, and hence, by Lemma~\ref{L:4.4}, we have
 $\alpha(M)\geq \alpha+m+1$. It remains to show that~\eqref{E:4.2}
 does in fact hold for all limit ordinals $\alpha$. This can be done
 by an easy induction argument. As we go from $\alpha$ to
 $\alpha+\omega$ in the induction step, we use the previous fact
 about finite ordinals. If $\alpha=\sup I$, where $I$ is the set of
 limit ordinals strictly smaller than $\alpha$, then
 \[
 \cN^{(\alpha)} =\bigcap_{\gamma\in I} \cN^{(\gamma)}\subset
 \bigcap_{\gamma\in I} \{N\in\cN:\,\alpha(N)\geq \gamma\}= \{
 N\in\cN:\,\alpha(N)\geq \alpha\}\ .
 \]
 We next establish the statement concerning $d$-diameters. Fix
 $M\in \cN$ and let $\alpha=\alpha(M)$. If
 $b\in M\setminus [T]^{(\alpha)}$, then $\alpha(b)<\alpha$. Thus
 $\alpha(N_b)<\alpha(M)$, and since $N_b\cap M$ contains $b$, we must
 have $N_b\subsetneq M$ using Lemma~\ref{L:4.3} and~\eqref{E:4.1}. It
 follows that $b\in N$ for some $N\in S_M$. We have proved that
 $M\setminus \bigcup_{N\in S_M}N \subset M\cap [T]^{(\alpha)}$, which
 shows that
 $d$-$\diam \Big(M\setminus \bigcup_{N\in S_M}N\Big)<\vp$.

 For the moreover part observe that $b\in N_b$ for all $b\in[T]$, and
 thus $[T]=\bigcup\cN$. Since $[T]$ is assumed to be compact, there
 is a finite cover of $[T]$ by some elements $N_1,N_2,\dots,N_k$ of
 $\cN$. By Lemma~\ref{L:4.3}, it follows that there can be at most
 $k$ initial nodes of $\cN$.
\end{proof}
\begin{lem}\label{L:4.5}
 Let $\cN$ be defined as in the proof of Theorem~\ref{T:4.2}. For
 $M\in\cN$ set $\Mt=M\setminus\bigcup _{N\in S_M} N$. Then the sets
 $\Mt$, $M\in\cN$, are pairwise disjoint, and for each $b\in [T]$
 there is an $M\in\cN$ so that $b\in\Mt$. Thus
 $\{\Mt:\,M\in\cN\}$ is a partition of $[T]$.
\end{lem}
\begin{proof}Let $M,N\in\cN$ with $\Mt\neq\Nt$. If
 $M\cap N=\emptyset$, then it is clear that $\Mt\cap\Nt=\emptyset$.
 Thus, by Lemma~\ref{L:4.3}, we may assume that $N\subsetneq M$. This
 means, since $\cN$ is a tree, that there is an $M'\in S_M$ with
 $N\subset M'$, which yields our first claim.

 Since $[T]=\bigcup\cN$, and since $\cN$ is a well-founded tree,
 there is a smallest (with respect to inclusion, or maximal in
 the order of $\cN$) $M\in \cN$ so that $b\in M$. This means that
 $b\notin N$ for any $N\in S_M$, which implies our
 second claim.
\end{proof}

By Lemma~\ref{L:4.5} we can define a map $q\colon [T]\to\cN$ by
letting $q(b)$ be the unique $M\in\cN$ such that $b\in \Mt$.
\begin{prop}
 \label{P:4.6} The map $q\colon [T]\to\cN$ defined above is onto. The
 quotient topology on $\cN$ induced by $q$ coincides with the tree
 topology of $\cN$.
\end{prop}
\begin{proof}
Let $M\in\cN$. We need to show that $\Mt\neq\emptyset$. Choose
$b\in[T]$ with $M=N_b$, and set $\alpha=\alpha(b)=\alpha(M)$. Let
$N\in S_M$. Then $\alpha(N)\leq\alpha$ by~\eqref{E:4.1}. If
$\alpha(N)<\alpha$, then $N$ is disjoint from $[T]^{(\alpha)}$, and
hence $b\notin N$. If $\alpha(N)=\alpha$, then $M$ is of type~II and
$N$ is of type~I by Lemma~\ref{L:4.4}. It follows that $b=b_t$ for
some $t\in T$, and $N\subset U_s$ for some $s\in S_t$.  But this
means that $t\not\in U_s$, and thus again, we have
$b=b_t\notin N$. This shows that $b\in \Mt$, and so $M=q(b)$ is in
the image of $q$.

We next observe that $q$ is continuous when $\cN$ is given the tree
topology. Indeed, let us fix $M\in\cN$ and $b\in [T]$, and set
$N=q(b)$. Then $b\in M$ if and only if $M\supset N$. Thus the
inverse image under $q$ of the basic clopen set $U_M$ (in the tree
topology of $(\cN, \supset)$, \ie, $U_M=\{N\in\cN: N\subset M\}$) in
$\cN$ is the clopen subset $M$ of $[T]$.  It follows that the
quotient topology of $\cN$ is finer than the tree topology. Since the
quotient topology is compact and the tree topology is Hausdorff, it
follows that these two topologies coincide, as claimed.
\end{proof}

\section{Zippin's theorem}
\label{S:5}
We now present our main result.
\begin{thm}
\label{T:5.1}
Let $X$ be an Asplund space, $(T,\preccurlyeq)$ be a tree with
finitely many initial nodes, and
$i\colon X\to C\big([T]\big)$ be an isometric embedding. Then for
all $\vp>0$ there exist a well-founded, compact tree $S$
with ordinal index
$\mathrm{o}(S)<\Sz(X,\vp/2)+\omega$ and an isometric copy $Y$ of
$C(S)$ in $C\big([T]\big)$ such that for all $x\in X$ there exists
$y\in Y$ with $\|i(x)-y\|\leq\vp\|x\|$.
\end{thm}
\begin{proof}
Consider the pseudo-metric $d$ on $[T]$ defined as follows.
\[
d(b,c)=\sup_{x\in B_X} |i(x)(b)-i(x)(c)|\ ,\qquad b,c\in
[T]\ .
\]
We identify $b\in [T]$ with its Dirac measure $\delta_b$.
Note that the dual map $i^*$ sends $[T]$ onto a $w^*$-closed,
1-norming subset of $B_{X^*}$,  and
\[
\|i^*(\delta_b)-i^*(\delta_c)\|_{X^*}=\sup_{x\in B_X}
|i(x)(b)-i(x)(c)|=d(b,c)\qquad \text{for all }b,c\in [T]\ .
\]
Fix $\vp>0$. It follows from above that the $\frac\vp2$-fragmentation
index of $[T]$ with respect to $d$ is equal to
$\Sz\big(i^*([T]),\frac\vp2\big)\le \Sz\big(X,\frac\vp2\big)$.
(See Example~\ref{Ex:4.1}.)

Theorem~\ref{T:4.2}, applied to $\frac\vp2$, provides us with a
well-founded, compact tree $\cN$ of basic clopen subsets of $[T]$
with 
$\mathrm{o}(\cN)<
\Frag\big(d,[T],\frac\vp2\big)+\omega=\Sz\big(i^*([T]),
\frac\vp2\big) + \omega\leq \Sz\big(X,\frac\vp2\big)+\omega$ such
that
\[
\text{$d$-$\diam$} \Big( M\setminus \bigcup_{N\in S_M} N \Big)
<\frac\vp2
\]
for all $M\in\cN$. By Proposition~\ref{P:4.6},  we also have a
quotient map $q\colon [T]\to \cN$, which is continuous with respect
to the tree topologies of $[T]$ and $\cN$. Thus, we have an
isometric embedding $q^*\colon C(\cN)\to C\big([T]\big)$ given by
$f\mapsto f\circ q$. Let $Y$ be the image of $q^*$. We will now show
that $Y$ is $\vp$-close to $i(X)$ which will prove the theorem
with $S=\cN$.

Let $x\in B_X$ and $g=i(x)$. Then $g$ is a continuous function on
$[T]$ whose oscillation on $\Mt=M\setminus \bigcup _{N\in S_M}N$ is
less than $\frac\vp2$ for all $M\in \cN$. Indeed, we
have
\[
|g(b)-g(c)|=|i(x)(b)-i(x)(c)|\leq d(b,c) <\frac\vp2\qquad \text{for
  all }b,c\in \Mt \text{ and }M\in\cN\ .
\]
For each $M\in\cN$ fix a point $x_M\in \Mt$ and set
$f_1(M)=g(x_M)$. This defines a function $f_1\colon\cN\to\R$. We will
now show that $f_1$ is not far from being continuous, and hence it is
not far from a continuous function.

Fix $M\in\cN$. Since the oscillation of $g$ on $\Mt$ is smaller
than $\frac\vp2$, it follows that
\[
\Big\{ (b,c)\in M\times M :\,|g(b)-g(c)|\ge \frac\vp2\Big\}\subset
\bigcup_{N\in S_M} (M\times N\cup N\times M)\ ,
\]
and thus by compactness of the left-hand set, there is a finite
set $F_M\subset S_M$ so that
\[
\Big\{ (b,c)\in M\times M:\,|g(b)-g(c)|\ge \frac\vp2\Big\}\subset
\bigcup_{N\in F_M} (M\times N\cup N\times M)\ ,
\]
and hence
\begin{equation}
  \label{E:5.1.1}
  |g(b)-g(c)|< \frac\vp2 \qquad\text{for all }b,c\in M\setminus
  \bigcup _{N\in F_M} N\ .
\end{equation}
Since for $P\in\cN$ if $P\in U_M\setminus \bigcup_{N\in F_M} U_N$,
then $x_P\in\Pt\subset M\setminus \bigcup_{N\in F_M} N$, it follows
from~\eqref{E:5.1.1} above that for all
$P,Q\in U_M\setminus \bigcup_{N\in F_M} U_N$, we have
\[
|f_1(P)-f_1(Q)|=|g(x_P)-g(x_Q)| <\frac\vp2\ .
\]
We have shown that in the compact space $\cN$ every point has a
neighbourhood on which the oscillation of the function $f_1$
is at most $\frac\vp2$. An application of
Lemma~\ref{lem:general-topology} now yields a
continuous function $f\colon\cN\to\R$ which is $\frac\vp2$-close to
$f_1$. We complete the proof by showing that $y=q^*(f)$ is
$\vp$-close to $g=i(x)$. Indeed, given $b\in [T]$, for $M=q(b)$ we
have $b,x_M\in\Mt$, and hence
\[
|y(b)-g(b)| = |f(M)-g(b)| \leq |f(M)-f_1(M)|+|g(x_M)-g(b)|\leq
\vp\ ,
\]
as required.
\end{proof}
From Theorem~\ref{T:5.1} the following (slight) sharpening of Zippin's
Theorem follows. Recall that $D$ denotes the Cantor set.
\begin{cor}
 \label{C:5.2}
 Let $X$ be a Banach space with separable dual and
 $i\colon X\to C(D)$ an isometric embedding (which always
 exists). Then for all $\vp>0$ there is a countable ordinal
 $\alpha<\omega^{\Sz(X,\vp/2)+\omega}$ and a subspace $Y$ of $C(D)$
 isometric to $C[0,\alpha]$ such that for all $x\in X$ there exists
 $y\in Y$ with $\|i(x)-y\|\leq \vp\|x\|$.
\end{cor}
\begin{proof}
 Since $X$ is separable, we can think of it as a subspace of
 $C(D)$. As explained in Example~\ref{Ex:2.6}, $D$ can be seen as the
 set of all branches of a tree $T$ endowed with the tree
 topology. Applying now Theorem~\ref{T:5.1}, we obtain a
 well-founded, compact tree $S$ with ordinal index
 $\mathrm{o}(S)<\Sz(X,\vp/2)+\omega$ and an isometric copy $Y$ of
 $C(S)$ in $C(D)$ such that for all $x\in X$ there exists $y\in Y$
 with $\|i(x)-y\|\leq\vp\|x\|$. It follows from the proof of
 Theorem~\ref{T:5.1} that $S$ is a subset of the basis $\cB$ of $D$
 consisting of clopen sets, so in particular $S$ is countable. By
 Theorem~\ref{thm:trees-vs-ordinals} and by the subsequent remark,
 there is a countable ordinal $\alpha$ such that $S$ is homeomorphic
 to $[0,\alpha]$, and moreover
 \[
 \alpha<\omega^{\mathrm{o}(S)}<\omega^{\Sz(X,\vp/2)+\omega}\ ,
 \]
 as claimed.
\end{proof}

We note the following corollary of Theorem~\ref{T:5.1}.
\begin{cor}
Let $T$ be a tree with finitely many initial nodes. If $T$ is
well-founded, then $C(T)$ is an Asplund space. Conversely, if
$C\big([T]\big)$ is an Asplund space, then there is a well-founded,
compact tree $S$ homeomorphic to $[T]$.
\end{cor}
\begin{proof}
 Assume that $T$ is well-founded. It follows that the ordinal index
 $\mathrm{o}(T)$ of $T$ exists. As explained in
 Remark~\ref{rem:szlenk-of-norming-set}, we can identify elements
 $t\in T$ with their Dirac measure $\delta_t$, and hence view $T$ as
 a $w^*$-closed, 1-norming subset of $B_{C(T)^*}$. Since
 $\|\delta_s-\delta_t\|=2$ for $s\neq t$ in $T$, it follows
 from~\eqref{E:2.1} for $0<\vp<2$ that
 $\Sz(T)=\Sz(T,\vp)=\CB(T)\le \mathrm{o}(T)$. Thus, in particular,
 $\Sz(T)\neq\infty$, and hence it follows from~\eqref{E:3.1} that
 $C(T)$ is Asplund.

 Now assume that $C([T])$ is Asplund. The we apply
 Theorem~\ref{T:5.1} to $X=C([T])\subset C([T])$ and $\vp=\frac12$
 and find a well-founded, compact tree $S$ and a closed subspace $Y$
 of $C\big([T]\big)$ as in the statement. Since now every element of
 $C\big([T]\big)$ has to be close to an element of $Y$, it follows
 that $Y=C\big([T]\big)$. Since $Y$ is isometric to $C(S)$, it
 follows from the Banach--Stone Theorem that $[T]$ is homeomorphic to
 $S$.
\end{proof}


\begin{thebibliography}{MM}


\bibitem{As} E.~Asplund, {\em Fr\'echet differentiability of convex
 functions,} Acta Math. {\bf 121} (1968), 31-- 47.

\bibitem{Be} Y. Benyamini, {\em An extension theorem for separable
 Banach spaces}, Israel J. of Math., {\bf 29} (1978), 24--30.


\bibitem{Ca} R. Causey,{\em The Szlenk index of injective tensor
 products and convex hulls.} J. Funct. Anal. {\bf 272} (2017), no. 8,
 3375 -- 3409.

\bibitem{Do} P.~Dodos, {\em Banach spaces and descriptive set theory:
 selected topics}. Lecture Notes in Mathematics. Springer-Verlag,
 Berlin (2010).


\bibitem{En} R.~Engelking, {\em General topology.}  Second
 edition. Sigma Series in Pure Mathematics, 6. Heldermann Verlag,
 Berlin, 1989.

\bibitem{FHHMZ} M. Fabian, P. Habala, P.  H\'ajek, V. Montesinos, and
 V. Zizler, {\it Banach Space Theory: The Basis for Linear and
   Nonlinear Analysis}, Canadian Math. Soc.  Books in Mathematics,
 Springer-Verlag, New York (2011).

\bibitem{Ke} A.~Kechris, {\em Classical descriptive set theory.}
 Graduate Texts in Mathematics, 156. Springer-Verlag, New York
 (1995).

\bibitem{NP} I. Namioka and R.~R.~Phelps, {\em Banach spaces which are
 Asplund spaces}, Duke Math. J. {\bf 42} (1968), 735 -- 750.

\bibitem{Phe} R.~R.~ Phelps, {\em Convex functions, monotone operators
 and differentiability}, Lecture Notes in Mathematics {\bf 1364},
 Springer (1989).


\bibitem{Sch} Th.~Schlumprecht, {\em On Zippin's embedding theorem of
 Banach spaces into Banach spaces with bases,} Adv. Math. {\bf 274}
 (2015), 833 -- 880.

\bibitem{Zi} M.~Zippin, {\em The separable extension problem}, Israel
 J. Math. 26 (1977), 372--387.

\end{thebibliography}
\end{document}